\documentclass[10pt,leqno]{amsart}
\usepackage{amscd}
\usepackage{color}
\usepackage{amssymb}
\usepackage{amsfonts}
\usepackage{latexsym}
\usepackage{verbatim}

\theoremstyle{plain}
\newtheorem{theorem}{Theorem}[section]

\newtheorem{prop}[theorem]{Proposition}

\newtheorem{rem}[theorem]{Remark}
\newtheorem{ex}[theorem]{Example}

\renewcommand{\b}{\begin{equation}}
\newcommand{\e}{\end{equation}}
\newcommand{\g}{\mathfrak{g}}
\newcommand\C{{\mathbb C}}

\title{A remark on the Bismut-Ricci form on $2$-step nilmanifolds}
\subjclass[2000]{Primary 53C15 ; Secondary 53B15, 53C30}
\thanks{This work was supported by  by G.N.S.A.G.A. of I.N.d.A.M} 
\address{Dipartimento di Matematica G. Peano \\ Universit\`a di Torino\\
Via Carlo Alberto 10\\
10123 Torino\\ Italy}
\email{mattia.pujia@unito.it, luigi.vezzoni@unito.it}
\author{Mattia Pujia and Luigi Vezzoni}
\date{\today}

\begin{document}
\maketitle
\begin{abstract}
In this note we observe that on a $2$-step nilpotent Lie group equipped with a left-invariant SKT structure the $(1,1)$-part of the Bismut-Ricci form is seminegative definite. As application we give a simplified proof of the non-existence of invariant SKT static metrics 
 on $2$-step nilmanifolds and of the existence of a long time solution to the pluriclosed flow in $2$-step nilmanifolds. 
\end{abstract}

\section{The Bismut Ricci form on $2$-step SKT nilmanifolds}
An Hermitian manifold is called SKT if its fundamental form is $\partial\bar \partial$-closed. The SKT condition can be described in terms of the Bismut connection by requiring that the torsion form is closed. Indeed, on any Hermitian manifold $(M,g)$ there is a  unique Hermitian connection $\nabla$ such that the tensor $c:= g(T(\cdot,\cdot),\dot )$ is skew-symmetric in its entries \cite{B}, where $T$ is the torsion  of $\nabla$. The metric $g$ is SKT if and only if $dc=0$. In this note we focus on the Ricci form of $\nabla$. In analogy to the 
K\"ahler case, the form is defined by 
$$
\rho^B(X,Y)= {\rm tr}_\omega R^B(X,Y,\cdot,\cdot)\,, 
$$
$\omega$ being the fundamental form of $g$ and $R^B$ the curvature tensor of $\nabla$. 

We consider as  manifold $M$ a $2$-step nilpotent Lie group $G$ equipped with an invariant Hermitian structure $(J,g)$. Under these assumptions, the form $\rho^B$ takes the following expression 
\begin{equation}\label{Fund}
\rho^B(X,Y)= i\sum_{r=1}^n g([X,Y],[Z_r,\bar Z_r])\,,\quad \mbox{ for every }X,Y\in \g\,,
\end{equation}
where $\{Z_r\}$ is an arbitrary $g$-unitary frame of the Lie algebra $\g$ of $G$ \footnote{Here we adopt the convection $\omega(\cdot,\cdot)=g(J\cdot,\cdot)$, in contrast to the one adopted in \cite{EFV2}.}. More generally, if $G$ is just a Lie group with an invariant Hermitian structure, $\rho^B$ takes the following expression 
\begin{equation}\label{general}
\rho^B(\omega)(X,Y) = - i \sum_{r=1}^n \left\lbrace g([[X,Y]^{1,0},Z_r],\bar Z_r) - g([[X,Y]^{0,1},\bar Z_r],Z_r ) -g([X,Y], [Z_r, \bar Z_r]) \right\rbrace \, .
\end{equation}

We have the following 
\begin{prop}\label{prop}
Let $G$ be a $2n$-dimensional $2$-step nilpotent Lie group with a left-invariant SKT structure $(J,g)$. Then 
$$
\rho^B(Z,\bar Z)=-i \sum_{r=1}^n \|[Z,\bar Z_r]\|^2
$$
for every $Z\in \mathfrak g^{1,0}$, where $\{Z_r\}$ is an arbitrary unitary frame. In particular, 
$$
\rho^B(X,JX)\leq 0
$$
for every  $X\in \g$.
\end{prop}

\begin{proof}
Let $Z$ and $W$ be vector fields of type $(1,0)$ on $\g\otimes \C$ and let $\omega$ be the fundamental form of $g$. Then we directly compute 
$$
\begin{aligned}
\partial\bar\partial\omega(Z,\bar Z,W,\bar W)=&-\bar\partial\omega([Z,\bar Z],W,\bar W)
+\bar\partial\omega([Z,W],\bar Z,\bar W)-\bar\partial\omega([Z,\bar W],\bar Z, W)\\
&-\bar\partial\omega([\bar Z, W],Z, \bar W)+\bar\partial\omega([\bar Z,\bar W], Z, W)
-\bar\partial\omega([W,\bar W], Z, \bar Z)\\
=&-\bar\partial\omega([Z,\bar Z]^{0,1},W,\bar W)
+\bar\partial\omega([Z,W],\bar Z,\bar W)-\bar\partial\omega([Z,\bar W]^{0,1},\bar Z, W)\\
&-\bar\partial\omega([\bar Z, W]^{0,1},Z, \bar W)+\bar\partial\omega([\bar Z,\bar W], Z, W)
-\bar\partial\omega([W,\bar W]^{0,1}, Z, \bar Z)\\
=&-\omega([Z,\bar Z]^{0,1},[W,\bar W]^{1,0})
+\omega([Z,W],[\bar Z,\bar W])-\omega([Z,\bar W]^{0,1},[\bar Z, W]^{1,0})\\
&-\omega([\bar Z, W]^{0,1},[Z, \bar W]^{1,0})+\omega([\bar Z,\bar W], [Z, W])
-\omega([W,\bar W]^{0,1},[Z, \bar Z]^{1,0})\\
=&+i g([Z,\bar Z]^{0,1},[W,\bar W]^{1,0})+ig([Z,W],[\bar Z,\bar W])+ig([Z,\bar W]^{0,1},[\bar Z, W]^{1,0})\\
&+ig([\bar Z, W]^{0,1},[Z, \bar W]^{1,0})-ig([\bar Z,\bar W], [Z, W])
+ig([W,\bar W]^{0,1},[Z, \bar Z]^{1,0})\\
=&+i g([Z,\bar Z],[W,\bar W])+ig([Z,\bar W],[\bar Z, W])\,.
\end{aligned}
$$
The SKT assumption $\partial\bar \partial \omega =0$ implies 
$$
g([Z,\bar Z],[W,\bar W])=-g([Z,\bar W],[\bar Z, W])\,. 
$$
Therefore in view of \eqref{Fund} we get 
$$
\rho^B(Z,\bar Z)= i\sum_{r=1}^n g([Z,\bar Z],[Z_r,\bar Z_r])=-i \sum_{r=1}^n g([Z,\bar Z_r],[\bar Z,Z_r]) \,,
$$
being $\{Z_r\}$ an arbitrary unitary frame, and the claim follows. 
\end{proof}

\begin{rem}
{\em Another description of the Bismut-Ricci form on $2$-step nilmanifolds can be found in \cite{AL}.}
\end{rem}

Next we  observe that in general the form  $\rho^B$ is not seminegative definite if we drop the the assumption on $G$ to be nilpotent or on the metric to be SKT.

\begin{ex}{\em 
Let $\g$ be the solvable unimodular Lie algebra with structure equations 
$$
de^1=0,\quad de^2=-e^{13}\,,\quad  de^3=e^{12},\quad de^4=-e^{23} \, ,
$$
equipped with the  complex structure $Je_1=e_4$ and $Je_2=e_3$ and the SKT metric 
$$
g= \sum_{r=1}^4e^{r}\otimes e^{r}+\frac12 (e^{1}\otimes e^{3}+e^{3}\otimes e^{1})-\frac{1}{2}(e^{2}\otimes e^{4}+e^{4}\otimes e^{2})  \, .
$$
By using \eqref{general} with respect to a unitary frame $\{Z_r\}$ we easily get 
$$
\rho^B=\frac23 e^{12}-\frac23 e^{13}+\frac43 e^{23} \,.
$$
In particular 
$$
\rho^B(e_2,Je_2) =  \frac43 \quad \mbox{ and }\quad  \rho^B(4e_1+e_2,J(4e_1+e_2)) = -\frac 43
$$
which implies that  $\rho^B$ is not seminegative definite as $(1,1)$-form.}
\end{ex}

\begin{ex}{\em
Let $(\g,J)$ be the 2-step nilpotent Lie algebra with structure equations
$$
de^1=de^2=de^3=0\,, \quad de^4=e^{12},\quad de^5=-e^{23},\quad de^6=e^{13} \, ,
$$ 
and equipped with the complex structure  $Je_1=e_2$, $Je_3=e_4$ and $Je_5=e_6$ and the non-SKT metric 
$$
g= \sum_{r=1}^6e^{r}\otimes e^{r}+\frac12 (e^{3}\otimes e^{6}+e^{6}\otimes e^{3})-\frac{1}{2}(e^{4}\otimes e^{5}+e^{5}\otimes e^{4})  \, .
$$
Again by using \eqref{general} with respect to a unitary frame $\{Z_r\}$ we easily get 
$$
\rho^B=-e^{12}-\frac12 e^{23} \,,
$$
which implies that  $\rho^B$ is not seminegative definite as $(1,1)$-form.
}
\end{ex}
\section{Non-existence of invariant SKT metrics  satisfying $(\rho^B)^{1,1}=\lambda \omega $} 

In this section  we observe that our proposition \ref{prop} easily implies that on a  $2$-step nilpotent Lie groups there are no SKT invariant metrics such that
$$
(\rho^B)^{1,1}=\lambda\omega 
$$ 
for some constant $\lambda$. This result is already known: the case $\lambda=0$ was studied in \cite{E}, while the case $\lambda\neq 0$ follows from \cite{EFV}.

Indeed in the setting of proposition \ref{prop}, if we assume $(\rho^B)^{1,1}=\lambda \omega $, then, taking into account that the center of $G$ in not trivial, formula \eqref{Fund} 
implies $\lambda=0$ and from proposition \ref{prop} it follows  $[\g^{1,0},\g^{0,1}]=0$. Therefore if $\{\zeta^k\}$ is a unitary co-frame in $\g$ we have  
$$
\bar \partial \zeta^k=0
$$
and we can write 
$$
\partial \zeta^k=c^{k}_{rs}\zeta^r\wedge \zeta^s\,.
$$
for some $c_{rs}^k$ in $\mathbb C$. 
Then 
$$
\partial \bar \partial \omega=i\partial \bar \partial \left(\sum_{k=1}^n\zeta^k\wedge \bar\zeta^k\right)=-i\partial  \left(\sum_{k=1}^n\bar{c}_{rs}^k\zeta^k\wedge \bar\zeta^{r}\wedge \bar \zeta^s\right)=-i\sum_{k=1}^nc^{k}_{ab}\bar{c}_{rs}^k\zeta^a\wedge \zeta^b\wedge \bar\zeta^{r}\wedge \bar \zeta^s
$$
and the SKT assumption implies that all the $c^k_{rs}$'s vanish in contrast to the assumption on $G$ to be not abelian.

\section{Long-time existence of the pluriclosed flow on $2$-step nilmanifolds}
The pluriclosed flow (PCF) is a parabolic flow of Hermitian metrics which preserves the SKT condition. The flow is defined on an SKT manifold $(M,\omega)$ as 
$$
\partial_t\omega_t=-(\rho^B_{\omega_t})^{1,1}\,,\quad \omega_{|t=0}=\omega\,,
$$
where $\rho^B_{\omega_t}$ is computed with respect to $\omega_t$ and the superscript \lq\lq $1,1$\rq\rq\, is the 
$(1,1)$-component with respect to $J$. The flow was introduced in \cite{streets-tian2} and then investigated in \cite{boling,streets-tian2,streets-tian4,Streets1,Streets} and it is a powerful tool in SKT geometry.   

In \cite{EFV2} it is proved that on a $2$-step nilpotent Lie group the flow has always a long-time solution for any initial invariant datum. The proof makes use of the bracket flow device introduce by Lauret in \cite{lauret}.  

In our setting, let $G$ be a $2$-step nilpotent Lie group with a left-invariant  complex structure $J$ and consider the PCF equation starting form an invariant SKT form $\omega$. The solution $\omega_t$ 
holds invariant for every $t$ and, therefore, the flow can be regarded as on ODE on $\Lambda^2\g^*\otimes \g$, where $\g$ is the Lie algebra of $G$. The bracket flow device consists in evolving the Lie bracket on $\g$ instead of the form $\omega$. For this purpose one considers the bracket variety $\mathcal A$ consisting on the elements $\lambda\in \Lambda^2\g^*\otimes \g$ such that 
\begin{eqnarray}
&& \lambda(\lambda(X,Y),V))=0\,,\\
&& \lambda(JX,JY)-J\lambda(JX,Y)-J\lambda(X,JY)-\lambda(X,Y)=0\,,\\
&& \partial_\lambda\bar \partial_\lambda\omega=0\,. 
\end{eqnarray} 
for every $X,Y,V\in \g$, where the operators $\partial_\lambda$ and $\bar \partial_\lambda$ are computed by using the bracket $\lambda$.
Any $\lambda\in \mathcal A$ gives a structure of $2$-step nilpotent Lie algebra to $\g$ such that $(J,\omega)$ is a SKT structure.  It turns out that the PCF is equivalent to a bracket flow type equation, i.e. an ODE in $\mathcal A$. The equivalence between the two equations is obtained by evolving the initial bracket $\mu$ of $\g$ as
$$
\mu_t(X,Y)=h_{t}\mu(h_t^{-1}X,h_t^{-1}Y)\,,\quad X,Y\in \g\,,
$$ 
being $h_t$ the curve in ${\rm End}(\g)$ solving 
$$
\frac{d}{dt}h_t=-\frac12 h_tP_{\omega_t}\,,\quad h_{|t=0}={\rm I}
$$
and $P_{\omega_t}\in {\rm End}(\g)$ is defined by
$$
\omega_t(P_{\omega_t}X,Y)=\frac12\left(\rho^B_{\omega_t}(X,Y)+\rho^B_{\omega_t}(JX,JY)\right)\,.
$$
The form $\omega_t$ reads in terms of $h_t$ as 
$$
\omega_t(X,Y)=\omega(h_tX,h_tY)\,. 
$$ 
Now in view of formula \eqref{Fund}
$$
\rho^B_{\omega_t}(X,\cdot)=0\,\mbox{ for every }X\in \xi
$$
and then $\omega_t(X,\cdot)=\omega(X,\cdot)$ for every $X\in \xi$, where $\xi$ is the center of $\mu$. 
Let $\xi^{\perp}$ be the $g$-orthogonal complement of $\xi$ in $\g$ and let $g_t$ be the Hermitian metric corresponding to the solution to the PCF equation starting from $\omega$. Then 
$$
\frac{d}{dt}g_t(X,\cdot)=0 \,\mbox{ for every }X\in \xi
$$
and $g_t$ preserves the splitting $\g=\xi\oplus \xi^\perp$ and the flow evolves only the component of $g$ in $\xi^\perp\times \xi^\perp$. It follows that $h_t$ preserves the splitting $\g=\xi\oplus \xi^{\perp}$ and 
$$
 h_{t|\xi}={\rm I}_{\xi}\,. 
$$
Since $(\g,\mu)$ is $2$-step nilpotent, then $\mu(X,Y)\in \xi$ for every $X,Y\in \g$ and 
$$
\mu_t(X,Y)=\mu(h_{t}^{-1}X,h_t^{-1}Y)\,.
$$
Therefore 
\begin{multline*}
\frac{d}{dt}\mu_t(X,Y)=-\mu(h_{t}^{-1}\dot{h}_t h_{t}^{-1}X,h_t^{-1}Y)-\mu(h_{t}^{-1}X,h_{t}^{-1}\dot{h}_t h_{t}^{-1}Y)\\
=-\mu_t(\dot{h}_t h_{t}^{-1}X,Y)-\mu_t(X,\dot{h}_t h_{t}^{-1}Y)=\frac12\mu_t(P_{\mu_t}X,Y)+\frac12\mu_t(X,P_{\mu_t}Y) \, ,
\end{multline*}
where for any $\lambda\in \mathcal A$ we set 
$$
\omega(P_\lambda X,Y)=i\frac12\sum_{r=1}^n\left( g(\lambda(X,Y),\lambda(Z_r,\bar Z_r))+g(\lambda(JX,JY),\lambda(Z_r,\bar Z_r))\right) 
$$
being $\{Z_r\}$ an arbitrary $g$-unitary frame and in the last step we have used 
$$
h_tP_{\omega_t}=P_{\mu_t}h_t\,.
$$
Hence the bracket flow equations writes as 
$$
\frac{d}{dt}\mu_t(X,Y)=\frac12\mu_t(P_{\mu_t}X,Y)+\frac12\mu_t(X,P_{\mu_t}Y)\,,\quad \mu_{|t=0}=\mu
$$
and its solution satisfies 
$$
\frac{d}{dt}g(\mu_t,\mu_t)=2g(\dot\mu_t,\mu_t)=4\sum_{r,s=1}^{2n}g(\mu_t(P_{\mu_t}e_r,e_s),\mu_t(e_r,e_s))
$$
being $\{e_r\}$ an arbitrary $g$-orthonormal frame. In view of proposition \ref{prop} all the eigenvalues of any $P_{\mu_t}$ are nonpositive. Fixing $t$ and taking as $\{e_r\}$ an orthonormal basis of eigenvectors of $P_{\mu_t}$ we get 
$$
\frac{d}{dt}g(\mu_t,\mu_t)=4\sum_{r,s=1}^{2n}a_r g(\mu_t(e_r,e_s),\mu_t(e_r,e_s))\leq 0\,.
$$ 
Therefore $\frac{d}{dt}g(\mu_t,\mu_t)\leq 0$ and the PCF is defined in $[0,\infty)$. 

\bigbreak\noindent{\it Acknowledgements.} The authors are grateful to Anna Fino for very useful conversations.


\begin{thebibliography}{12}

\bibitem{AL} R. Arroyo and R. Lafuente, The pluriclosed flow on nilmanifold and tamed symplectic forms, in preparation. 
\bibitem{B}   J.-M. Bismut, A local index theorem for non-K\"ahler manifolds.  {\em Math. Ann.} {\bf 284} (1989), no. 4, 681--699.


\bibitem{boling}
J. Boling, Homogeneous Solutions of Pluriclosed Flow on Closed Complex Surfaces. {\em J. Geom. Anal.} {\bf 26} (2016), no. 3, 2130--2154.

\bibitem{E}
N. Enrietti, Static SKT metrics on Lie groups,  {\em Manuscripta Math.} {\bf 140} (2013), no. 3--4, 557--571.

\bibitem{EFV}
N. Enrietti, A. Fino and L. Vezzoni, Tamed symplectic forms and strong K\"ahler with torsion metrics. 
{\em J. Symplectic Geom. } {\bf 10}, n. 2 (2012), 203--223.

\bibitem{EFV2} 
N. Enrietti, A. Fino and L. Vezzoni,
The pluriclosed flow on nilmanifolds and Tamed symplectic forms, {\em J. Geom. Anal.} {\bf 25} (2015), no. 2, 883--909.   


\bibitem{finosalamonparton}
A. Fino, M. Parton and  S. M. Salamon, Families of strong KT structures in six dimensions. {\em Comment. Math. Helv.} {\bf 79} (2004), no. 2, 317--340.

\bibitem{Gaudbumi}
P. Gauduchon, Hermitian connections and Dirac operators. {\em Boll. Un. Mat. Ital. B (7)} {\bf 11} (1997), no. 2, suppl., 257--288.

\bibitem{lauret}
J. Lauret, The Ricci flow for simply connected nilmanifolds. {\em Comm. Anal. Geom.} {\bf 19} (2011), no. 5, 831--854.

\bibitem{streets-tian2}
J. Streets and  G. Tian, A parabolic flow of pluriclosed metrics. {\em Int. Math. Res. Notices} (2010), 3101--3133.

\bibitem{streets-tian4}
J. Streets and  G. Tian, Regularity results for pluriclosed flow. {\em Geom. Topol.} {\bf 17} (2013), no. 4, 2389--2429\,.

\bibitem{Streets1}  
J. Streets, Pluriclosed flow, Born-Infeld geometry, and rigidity results for generalized K\"ahler manifolds. {\em Comm. Partial Differential Equations} {\bf 41} (2016), no. 2, 318--374. 

\bibitem{Streets}  
J. Streets, Pluriclosed flow on manifolds with globally generated bundles, {\em Complex Manifolds} {\bf 3} (2016), 222--230. 

\bibitem{luigiproc}  L. Vezzoni, A note on canonical Ricci forms on 2-step nilmanifolds. {\em Proc. Amer. Math. Soc.} {\bf 141} (2013), no. 1, 325--333.

\end{thebibliography}
\end{document}